\documentclass[12pt,twoside]{amsart}
\usepackage{amssymb,amsmath,amsthm, amscd, enumerate, mathrsfs}
\usepackage{graphicx, hhline}
\usepackage[all]{xy}
\usepackage[dvipdfmx]{hyperref}

\title{On normalization of quasi-log canonical pairs}
\author{Osamu Fujino and Haidong Liu}
\date{2018/8/20, version 0.17}
\address{Department of Mathematics, Graduate School of Science, 
Osaka University, Toyonaka, Osaka 560-0043, Japan}
\email{fujino@math.sci.osaka-u.ac.jp}
\address{Department of Mathematics, Graduate School of Science, 
Kyoto University, Kyoto 606-8502, Japan}
\email{liu.dong.82u@st.kyoto-u.ac.jp}
\DeclareMathOperator{\Nqklt}{Nqklt}
\newtheorem{thm}{Theorem}[section]

\newtheorem{prop}[thm]{Proposition}
\newtheorem{conj}[thm]{Conjecture}
\newtheorem{cor}[thm]{Corollary}
\newtheorem{claim}{Claim}

\theoremstyle{definition}
\newtheorem{ex}[thm]{Example}
\newtheorem{defn}[thm]{Definition}
\newtheorem{rem}[thm]{Remark}
\newtheorem*{ack}{Acknowledgments}       

\makeatletter
    
    \@addtoreset{equation}{section}
\makeatother

\begin{document}

\maketitle 

\begin{abstract}
The normalization of an irreducible quasi-log canonical pair naturally becomes 
a quasi-log canonical pair. 
\end{abstract}

\tableofcontents

\section{Introduction}\label{f-sec1}
In \cite{ambro}, Florin Ambro introduced 
the notion of {\em{quasi-log varieties}}, which are now called {\em{quasi-log schemes}}, 
in order to establish the cone and contraction theorem for {\em{generalized 
log varieties}}. Note that a generalized 
log variety is a pair $(X, \Delta)$ consisting of a normal irreducible variety 
$X$ and an effective $\mathbb R$-divisor $\Delta$ on $X$ 
such that $K_X+\Delta$ is $\mathbb R$-Cartier. 
Although the main result of \cite{ambro} was recovered 
without using the theory of quasi-log schemes in 
\cite{fujino-fund}, it became clear that quasi-log schemes are 
ubiquitous in the theory of minimal models (see, for example, 
\cite{fujino-slc} and \cite{fujino-foundations}). 
As Ambro said in \cite{ambro}, 
the definition of quasi-log schemes is motivated by Kawamata's X-method. 
Therefore, it is not surprising that quasi-log schemes often appear naturally 
in the theory of minimal models. 
In this paper, we prove that the normalization of an irreducible 
{\em{quasi-log canonical pair}} ({\em{qlc pair}}, for short) 
becomes a quasi-log canonical pair. Note that the notion of 
quasi-log canonical pairs is one of the useful generalizations of log canonical pairs 
in the framework of quasi-log schemes. In general, a 
quasi-log canonical pair may be reducible and may not necessarily be equidimensional. 
We also note that the result of this paper plays a crucial role 
when we show that every quasi-log canonical pair has only 
Du Bois singularities in \cite{fujino-haidong}.

\medskip

Let $(X, \Delta)$ be a log canonical 
pair and let $W$ be a log canonical center of $(X, \Delta)$. 
Then $[W, \omega]$ has a natural qlc structure, where 
$\omega=(K_X+\Delta)|_W$. 
For the details, see Example \ref{z-ex2.11} below and 
\cite[6.4.1 and 6.4.2]{fujino-foundations}.  
Let $\nu: W^\nu\to W$ be the normalization. 
Then we expect that there exists an effective $\mathbb R$-divisor 
$\Delta_{W^\nu}$ on $W^\nu$ such that 
$(W^\nu, \Delta_{W^\nu})$ is log canonical 
and that $K_{W^\nu}+\Delta_{W^\nu}\sim 
_{\mathbb R} \nu^*\omega$. 
However, it is still a difficult open problem to find $\Delta_{W^\nu}$ 
with the above properties. 
For a related topic, see \cite{fujino-gongyo}. 
By Theorem \ref{f-thm1.1} below, which is the main theorem of this paper,  
we see that $[W^\nu, \nu^*\omega]$ naturally becomes a qlc pair. 
Therefore, we can apply 
the theory of quasi-log schemes to 
$[W^\nu, \nu^*\omega]$. 

\begin{thm}[Normalization of qlc pairs]\label{f-thm1.1}
Let $[X, \omega]$ be a qlc pair such that 
$X$ is irreducible. 
Let $\nu:Z\to X$ be the normalization. 
Then $[Z, \nu^*\omega]$ naturally becomes a qlc pair with 
the following properties: 
\begin{itemize}
\item[(i)] if $C$ is a qlc center of $[Z, \nu^*\omega]$, 
then $\nu(C)$ is a qlc center of $[X, \omega]$, 
and 
\item[(ii)] $\Nqklt(Z, \nu^*\omega)=\nu^{-1}(\Nqklt(X, \omega))$. 
More precisely, the equality $$\nu_*\mathcal I_{\Nqklt(Z, \nu^*\omega)}=
\mathcal I_{\Nqklt(X, \omega)}$$ holds, where $
\mathcal I_{\Nqklt(X, \omega)}$ 
and $\mathcal I_{\Nqklt(Z, \nu^*\omega)}$ are the defining ideal sheaves of 
$\Nqklt(X, \omega)$ and $\Nqklt(Z, \nu^*\omega)$ respectively. 
\end{itemize}  
\end{thm}

For the definition of qlc pairs and $\Nqklt(X, \omega)$, see 
Definitions \ref{z-def2.4} and \ref{z-def2.7}, respectively. 
By the theory of quasi-log schemes discussed 
in \cite[Chapter 6]{fujino-foundations} and Theorem \ref{f-thm1.1}, 
the fundamental theorems of the minimal model program hold 
for $[Z, \nu^*\omega]$. More precisely, 
the cone and contraction theorem and the basepoint-free 
theorem of Reid--Fukuda type hold for $[Z, \nu^*\omega]$ by 
\cite[Theorem 6.4.7]{fujino-foundations} and 
\cite[Theorem 6.9.1]{fujino-foundations} respectively 
(see also \cite{fujino-reid-fukuda}). 
We can also apply various vanishing theorems to $[Z, \nu^*\omega]$. 
As a special case, we have the following vanishing theorem. 

\begin{cor}[Vanishing theorem for 
normalizations]\label{f-cor1.2}
We use the same notation as in Theorem \ref{f-thm1.1}. 
Let $\pi:X\to S$ be a proper morphism onto a scheme $S$ and  
let $L$ be a Cartier divisor on $X$ such that $L-\omega$ is 
nef and log big over $S$ with respect to $[X, \omega]$. 
Then 
$$
R^i(\pi\circ \nu)_*\mathcal O_Z(\nu^*L)=0
$$ 
for every $i>0$. 
\end{cor}

Let us discuss some conjectures for qlc pairs. 
The second author poses the following conjecture on 
Du Bois singularities. 

\begin{conj}\label{f-conj1.3} Let $[X, \omega]$ be a qlc pair. 
Then $X$ has only Du Bois singularities. 
\end{conj}

The statement of Conjecture 
\ref{f-conj1.3} is a complete 
generalization of \cite[Corollary 6.32]{kollar}. 
For the details of Du Bois singularities, 
see \cite[Section 5.3]{fujino-foundations} and 
\cite[Chapter 6]{kollar}.  
By Theorem \ref{f-thm1.1}, 
we have: 

\begin{prop}\label{f-prop1.4}
It is sufficient to prove Conjecture \ref{f-conj1.3} 
under the extra assumption that 
$X$ is normal. 
\end{prop}

Finally, we pose the following conjecture on normal qlc pairs. 

\begin{conj}\label{f-conj1.5}
Let $[X, \omega]$ be a qlc pair such that $X$ is quasi-projective and 
normal. 
Then there exists an effective $\mathbb Q$-divisor $\Delta$ on 
$X$ such that $(X, \Delta)$ is log canonical. 
\end{conj}

We note that \cite[Theorem 1.1]{fujino-some} strongly supports 
Conjecture \ref{f-conj1.5}. 
Of course, Conjecture \ref{f-conj1.3} follows from Conjecture 
\ref{f-conj1.5} by Proposition \ref{f-prop1.4}. 
This is because log canonical singularities are known to be Du Bois 
(see \cite{kollar}). 

\medskip 

Although Theorem \ref{f-thm1.1} 
may look somewhat artificial, 
it plays an important role 
in \cite{fujino-wenfei1}, \cite{fujino-wenfei2}, 
and \cite{fujino-haidong}. 
Roughly speaking, in \cite{fujino-slc-trivial}, 
we prove that $X$ is {\em{generalized lc}} in the sense of 
Birkar--Zhang (see \cite{birkar-zhang}) 
with some good properties when $[X, \omega]$ is a 
normal irreducible quasi-log canonical pair. 
We can see it as a weak solution of Conjecture \ref{f-conj1.5}. 
We note that \cite{fujino-slc-trivial} 
heavily depends on the theory of 
variations of mixed Hodge structure on cohomology 
with compact support (see \cite{fujino-fujisawa}). 
Then, in \cite{fujino-haidong}, 
we completely confirm Conjecture \ref{f-conj1.3}, 
that is, we show that 
$X$ has only Du Bois singularities if $[X, \omega]$ is a 
quasi-log canonical pair. 

\medskip 

We will work over $\mathbb C$, the complex number 
field, throughout this paper. 
A {\em{scheme}} means a separated scheme of finite type 
over $\mathbb C$. 
A {\em{variety}} means a reduced scheme, that is, 
a reduced separated scheme of finite type over $\mathbb C$. 
We will freely use the basic notation of the minimal model 
program as in \cite{fujino-fund}, \cite{fujino-slc}, and 
\cite{fujino-foundations}. 
For the details of the theory of quasi-log schemes, 
we recommend the reader to see \cite[Chapter 6]{fujino-foundations}. 

\begin{ack}
The first author was partially 
supported by JSPS KAKENHI Grant Numbers JP16H03925, JP16H06337. 
The authors would like to thank Professor Wenfei Liu, whose 
question is one of the motivations of this paper. 
\end{ack}

\section{Quick review of the theory of quasi-log schemes}

In this section, we quickly review the theory of quasi-log schemes 
because it is not so popular yet. 

\medskip 

Before we explain the definition of quasi-log canonical pairs, 
we prepare some basic definitions. 

\begin{defn}[$\mathbb R$-divisors]\label{z-def2.1}
Let $X$ be an equidimensional variety, 
which is not necessarily regular in codimension one. 
Let $D$ be an $\mathbb R$-divisor, 
that is, $D$ is a finite formal sum $\sum _i d_i D_i$, where 
$D_i$ is an irreducible reduced closed subscheme of $X$ of 
pure codimension one and $d_i$ is a real number for every $i$ 
such that $D_i\ne D_j$ for $i\ne j$. 
We put 
\begin{equation*}
D^{<1} =\sum _{d_i<1}d_iD_i, \quad 
D^{\leq 1}=\sum _{d_i\leq 1} d_i D_i, \quad 
\text{and} \quad
\lceil D\rceil =\sum _i \lceil d_i \rceil D_i, 
\end{equation*}
where $\lceil d_i\rceil$ is the integer defined by $d_i\leq 
\lceil d_i\rceil <d_i+1$. 

Let $B_1$ and $B_2$ be $\mathbb R$-Cartier divisors on $X$. 
Then $B_1\sim _{\mathbb R} B_2$ means that 
$B_1$ is $\mathbb R$-linearly equivalent 
to $B_2$. 
\end{defn}

We note that we can define {\em{$\mathbb Q$-divisors}} 
and $\sim_{\mathbb Q}$ similarly. 

\medskip 

The notion of {\em{globally embedded simple normal crossing 
pairs}} play a crucial role in the theory of quasi-log schemes 
described in \cite[Chapter 6]{fujino-foundations}. 

\begin{defn}[Globally embedded simple normal 
crossing pairs]\label{z-def2.2} 
Let $Y$ be a simple normal crossing 
divisor on a smooth variety $M$ and let $B$ be 
an $\mathbb R$-divisor 
on $M$ such that 
$Y$ and $B$ have no common irreducible components and 
that the support of $Y+B$ is a simple normal crossing divisor on $M$. In this 
situation, $(Y, B_Y)$, where $B_Y:=B|_Y$, 
is called a {\em{globally embedded simple 
normal crossing pair}}.
\end{defn}

\begin{defn}[Strata of simple normal crossing divisors]\label{z-def2.3}
Let $Y$ be a simple normal crossing divisor on a smooth variety 
and let $Y=\bigcup_{i\in I}Y_i$ be the irreducible decomposition of $Y$. 
A {\em{stratum}} of $Y$ is an irreducible component of 
$Y_{i_1}\cap \cdots \cap Y_{i_k}$ for some $\{i_1, \ldots, i_k\}\subset I$. 
\end{defn}

Let us recall the definition of {\em{quasi-log canonical pairs}}. 

\begin{defn}[Quasi-log canonical pairs]\label{z-def2.4}
Let $X$ be a scheme and let $\omega$ be an 
$\mathbb R$-Cartier divisor (or an $\mathbb R$-line bundle) on $X$. 
Let $f:Y\to X$ be a proper morphism from a globally embedded 
simple normal 
crossing pair $(Y, B_Y)$. If $B_Y$ is a 
subboundary $\mathbb R$-divisor, 
that is, $B_Y=B^{\leq 1}_Y$, 
$f^*\omega\sim _{\mathbb R} K_Y+B_Y$ holds, and the natural map 
$$\mathcal O_X\to f_*\mathcal O_Y(\lceil -(B_Y^{<1})\rceil)$$ is an 
isomorphism, 
then $\left(X, \omega, f:(Y, B_Y)\to X\right)$ 
is called a {\em{quasi-log canonical pair}} 
({\em{qlc pair}}, for short). If there is no danger of confusion, 
we simply say that $[X, \omega]$ is a qlc pair. 
\end{defn}

The notion of {\em{qlc strata}} and {\em{qlc centers}} 
is very important. It is indispensable for inductive 
treatments of quasi-log canonical pairs. 

\begin{defn}[Qlc strata and qlc centers]\label{z-def2.5} 
Let $\left(X, \omega, f:(Y, B_Y)\to X\right)$ be a quasi-log canonical 
pair as in 
Definition \ref{z-def2.4}. 
Let $\nu:Y^\nu\to Y$ be the normalization. 
We put $$K_{Y^\nu}+\Theta=\nu^*(K_Y+B_Y),$$ that is, 
$\Theta$ is the sum of the inverse images of $B_Y$ and 
the singular locus of $Y$. 
Then $(Y^\nu, \Theta)$ is sub log canonical in the usual sense. 
Let $W$ be a log canonical center of $(Y^\nu, \Theta)$ or 
an irreducible component of $Y^\nu$. 
Then $f\circ \nu(W)$ is called a {\em{qlc stratum}} of 
$\left(X, \omega, f:(Y, B_Y)\to X\right)$. 
If there is no danger of confusion, we simply call it 
a qlc stratum of 
$[X, \omega]$. 
If $C$ is a qlc stratum of $[X, \omega]$ but it is not 
an irreducible component of $X$, then $C$ is called 
a {\em{qlc center}} of $\left(X, \omega, f:(Y, B_Y)\to X\right)$ 
or simply of $[X, \omega]$. 
\end{defn}

One of the most important results in the theory of 
quasi-log schemes is {\em{adjunction}}. 

\begin{thm}[{Adjunction, see \cite[Theorem 6.3.5]{fujino-foundations}}]
\label{z-thm2.6}
Let $[X, \omega]$ be a qlc pair and let $X'$ be 
the union of some qlc strata of $[X, \omega]$. 
Then $[X', \omega|_{X'}]$ is a qlc pair 
such that 
the qlc strata of $[X', \omega|_{X'}]$ are exactly 
the qlc strata of $[X, \omega]$ that are contained in $X'$. 
\end{thm}

We strongly recommend the reader to see \cite[Theorem 6.3.5]
{fujino-foundations} and its proof for the details of Theorem \ref{z-thm2.6}. 
Theorem \ref{z-thm2.6} is a special case of \cite[Theorem 6.3.5 (i)]
{fujino-foundations}. 

\begin{defn}[Union of all qlc centers]\label{z-def2.7} 
Let $[X, \omega]$ be a qlc pair. 
The union of all qlc centers of $[X, \omega]$ is denoted by 
$\Nqklt (X, \omega)$. It is very important that 
$$[\Nqklt (X, \omega), \omega|_{\Nqklt(X, \omega)}]$$ has a quasi-log canonical 
structure induced from $\left(X, \omega, f:(Y, B_Y)\to X\right)$ 
by adjunction (see Theorem \ref{z-thm2.6} and \cite[Theorem 6.3.5 (i)]
{fujino-foundations}). 
\end{defn}

The vanishing theorem is also a very important result. 
Theorem \ref{z-thm2.8} is a special case 
of \cite[Theorem 6.3.5 (ii)]{fujino-foundations}. 

\begin{thm}[{Vanishing theorem, see \cite[Theorem 6.3.5]{fujino-foundations}}]
\label{z-thm2.8}
Let $[X, \omega]$ be a qlc pair and let $\pi:X\to S$ be 
a proper morphism between schemes. 
Let $L$ be a Cartier divisor on $X$ such that 
$L-\omega$ is nef and log big over $S$ with respect to $[X, \omega]$, 
that is, $L-\omega$ is $\pi$-nef and $(L-\omega)|_W$ is 
$\pi$-big for every qlc stratum $W$ of $[X, \omega]$. 
Then $R^i\pi_*\mathcal O_X(L)=0$ for 
every $i>0$. 
\end{thm}

The notion of {\em{$\mathbb Q$-structures}} is 
introduced in \cite{fujino-slc-trivial}. 

\begin{defn}[$\mathbb Q$-structures]\label{z-def2.9}
If $\omega$ is a $\mathbb Q$-Cartier 
divisor (or a $\mathbb Q$-line bundle) on $X$, $B_Y$ is a $\mathbb Q$-divisor 
on $Y$, 
and $f^*\omega\sim _{\mathbb Q} K_Y+B_Y$ holds in 
Definition \ref{z-def2.4}, 
then we say that $\left(X, \omega, f:(Y, B_Y)\to X\right)$ 
has a $\mathbb Q$-structure or simply say that $[X, \omega]$ has a 
$\mathbb Q$-structure. 
\end{defn}

\begin{rem}\label{z-rem2.10} 
If $[X, \omega]$ has a $\mathbb Q$-structure, 
then we can easily see that for any union of 
qlc strata $X'$ the qlc pair $[X', \omega|_{X'}]$ 
naturally has a $\mathbb Q$-structure in 
Theorem \ref{z-thm2.6}. 
For the details, see the proof of \cite[Theorem 6.3.5]{fujino-foundations}. 
\end{rem}

We close this section with an important example. 

\begin{ex}\label{z-ex2.11}
Let $(X, \Delta)$ be a log canonical pair. 
We put $\omega=K_X+\Delta$. 
Let $f:Y\to X$ be a resolution 
such that $K_Y+B_Y=f^*(K_X+\Delta)$. 
We assume that the support of $B_Y$ is a simple normal crossing 
divisor on $Y$. 
Then we can easily see that $(Y, B_Y)$ is a globally embedded 
simple normal crossing pair, 
$B_Y=B^{\leq 1}_Y$, and the natural map 
$$
\mathcal O_X\to f_*\mathcal O_Y(\lceil (-B^{<1}_Y)\rceil)
$$ 
is an isomorphism. 
Therefore, we can see that 
$\left(X, \omega, f:(Y, B_Y)\to X\right)$ is a qlc 
pair. 
In this situation, $W$ is a qlc stratum of $[X, \omega]$ 
if and only if $W$ is a log canonical center of $(X, \Delta)$ or 
$X$ itself. 
\end{ex}

Anyway, we recommend the reader to see \cite[Chapter 6]{fujino-foundations} 
for the theory of quasi-log schemes. 

\section{Proof}\label{f-sec3}

Let us start the proof of Theorem \ref{f-thm1.1}. 

\begin{proof}[Proof of Theorem \ref{f-thm1.1}]
Let $f:(Y, B_Y)\to X$ be a proper surjective morphism from a globally 
embedded simple normal crossing pair $(Y, B_Y)$ as in 
Definition \ref{z-def2.4}. 
By \cite[Proposition 6.3.1]{fujino-foundations}, 
we may assume that 
the union of all strata of $(Y, B_Y)$ mapped to $\Nqklt(X, \omega)$, 
which is denoted by $Y''$, is a union of some irreducible 
components of $Y$. 
We put $Y'=Y-Y''$ and $K_{Y'}+B_{Y'}=(K_Y+B_Y)|_{Y'}$. 
Then we obtain the following commutative diagram: 
$$
\xymatrix{
Y' \ar[d]_{f'}\ar@{^(->}[r]^\iota&Y\ar[d]^f\\ 
V \ar[r]_p& X
}
$$
where $\iota:Y'\to Y$ is a natural closed immersion 
and 
$$
\xymatrix{Y' \ar[r]^{f'}& V \ar[r]^p& X
}
$$ 
is the Stein factorization of $f\circ \iota:Y'\to X$. 
By construction, 
$\iota:Y'\to Y$ is an isomorphism 
over the generic point of $X$.
By construction again, the natural map $\mathcal O_V\to 
f'_*\mathcal O_{Y'}$ is an isomorphism 
and every stratum of $Y'$ is dominant onto $V$. 
Therefore, $p$ is birational. 
\begin{claim}\label{f-claim1}
$V$ is normal. 
\end{claim}
\begin{proof}[Proof of Claim \ref{f-claim1}] 
(cf.~the proof of \cite[Lemma 6.3.9]{fujino-foundations}). 
Let $\pi:V^n\to V$ be the normalization. 
Since every stratum of $Y'$ is dominant onto $V$, 
there exists a closed subset $\Sigma$ of $Y'$ such that 
$\mathrm{codim}_{Y'}\Sigma\geq 2$ and that 
$\pi^{-1}\circ f': Y'\dashrightarrow V^n$ is a morphism 
on $Y'\setminus \Sigma$. 
Let $\widetilde Y$ be the graph of 
$\pi^{-1}\circ f': Y'\dashrightarrow V^n$. 
Then we have the following commutative 
diagram: 
$$
\xymatrix{
\widetilde Y \ar[d]_-{\widetilde f}\ar[r]^-q&Y'\ar[d]^-{f'}\\ 
V^n\ar[r]_-\pi& V
}
$$
where $q$ and $\widetilde f$ are natural projections. 
Note that $q:\widetilde Y\to Y'$ is an isomorphism 
over $Y\setminus \Sigma$ by construction. 
Since $Y'$ is a simple normal crossing divisor on a smooth variety and 
$\mathrm{codim}_{Y'}\Sigma\geq 2$, the natural map 
$\mathcal O_{Y'}\to q_*\mathcal O_{\widetilde Y}$ is an isomorphism. 
Therefore, the composition 
$$
\mathcal O_V\to \pi_*\mathcal O_{V^n}\to \pi_*\widetilde f_*\mathcal O_{\widetilde Y}
=f'_*q_*\mathcal O_{\widetilde Y}\simeq \mathcal O_V
$$ 
is an isomorphism. 
Thus we have $\mathcal O_V\simeq \pi_*\mathcal O_{V^n}$. 
This implies that $V$ is normal. 
\end{proof} 
Therefore, $p:V\to X$ is nothing but the normalization $\nu:Z\to X$. 
So we have the following commutative diagram. 
$$
\xymatrix{
Y' \ar[d]_{f'}\ar@{^(->}[r]^\iota&Y\ar[d]^f\\ 
Z \ar[r]_\nu& X
}
$$
\begin{claim}\label{f-claim2}The natural map 
$$
\alpha: \mathcal O_Z\to f'_*\mathcal O_{Y'}
(\lceil -(B^{<1}_{Y'})\rceil)
$$ 
is an isomorphism. 
\end{claim}
\begin{proof}[Proof of Claim \ref{f-claim2}]
Note that $\nu: Z\to X$ is an isomorphism 
over $X\setminus \Nqklt (X, \omega)$. 
Therefore, $\alpha$ is an isomorphism outside 
$\nu^{-1}(\Nqklt (X, \omega))$. 
Since $Z$ is normal and $f'_*\mathcal O_{Y'}(\lceil -(B^{<1}_{Y'})
\rceil)$ is torsion-free, it is sufficient to see that $\alpha$ is an isomorphism 
in codimension one. 
Let $P$ be any prime divisor on $Z$ such that 
$P\subset \nu^{-1}(\Nqklt (X, \omega))$. 
Then, by construction, there exists an irreducible 
component of $B^{=1}_{Y'}$ which maps onto $P$. 
We note that every fiber of $f$ is connected by $f_*\mathcal O_Y
\simeq \mathcal O_X$. 
Therefore, the effective divisor $\lceil -(B^{<1}_{Y'})\rceil$ 
does not contain the whole fiber of $f'$ over the generic 
point of $P$. Thus, 
$\alpha$ is an isomorphism at the generic point of $P$. 
This means that $\alpha$ is an isomorphism. 
\end{proof}
Therefore, by Claim \ref{f-claim2}, 
$f':(Y', B_{Y'})\to Z$ defines a quasi-log structure on $[Z, \nu^*\omega]$. 
By construction, the property (i) automatically holds. 
Let us consider the following ideal sheaf: 
$$
\mathcal I=f'_*\mathcal O_{Y'}(\lceil -(B^{<1}_{Y'})\rceil 
-Y''|_{Y'})\subset 
f'_*\mathcal O_{Y'}(\lceil -(B^{<1}_{Y'})\rceil)=\mathcal O_Z. 
$$ 
We note that $\mathcal I=\mathcal I_{\Nqklt (Z, \nu^*\omega)}$ 
since $\Nqklt(Z, \nu^*\omega)=f'(Y''|_{Y'})$. 
\begin{claim}\label{f-claim3}
$f_*\mathcal O_{Y'}(\lceil -(B^{<1}_{Y'})\rceil -Y''|_{Y'})=\mathcal I_{\Nqklt (X, \omega)}$ holds. 
\end{claim}
\begin{proof}[Proof of Claim \ref{f-claim3}] 
(cf.~the proof of \cite[Theorem 6.3.5 (i)]{fujino-foundations}). 
Since $\mathcal O_{Y'}(\lceil -(B^{<1}_{Y'})\rceil-Y''|_{Y'})\subset 
\mathcal O_Y(\lceil -(B^{<1}_Y)\rceil)$, 
we get 
$$
f_*\mathcal O_{Y'}(\lceil -(B^{<1}_{Y'})\rceil -Y''|_{Y'})\subset 
f_*\mathcal O_Y(\lceil -(B^{<1}_Y)\rceil)=\mathcal O_X, 
$$ 
that is, $f_*\mathcal O_{Y'}(\lceil -(B^{<1}_{Y'})\rceil -Y''|_{Y'})$ is an ideal 
sheaf on $X$. 
By construction, 
$$f_*\mathcal O_{Y'}(\lceil -(B^{<1}_{Y'})\rceil -Y''|_{Y'})=\mathcal I_{\Nqklt (X, \omega)}
$$ holds. 
Here, we used the fact that every fiber of $f$ is connected. 
\end{proof}
Claim \ref{f-claim3} implies that 
$$\nu_*\mathcal I=\nu_*f'_*\mathcal O_{Y'}(\lceil -(B^{<1}_{Y'})\rceil -Y''|_{Y'})
=f_*\mathcal O_{Y'}(\lceil -(B^{<1}_{Y'})\rceil -Y''|_{Y'})=\mathcal I_{\Nqklt (X, \omega)}. 
$$ 
Since $\nu$ is finite, 
$\mathcal I=\nu^{-1}\mathcal I_{\Nqklt (X, \omega)}\cdot \mathcal O_Z$. 
Therefore, we have 
$\nu^{-1}(\Nqklt (X, \omega))=\Nqklt (Z, \nu^*\omega)$. 
This means that (ii) holds. 
\end{proof}

\begin{proof}[Proof of Corollary \ref{f-cor1.2}]
This follows from Theorems \ref{f-thm1.1} and 
\ref{z-thm2.8} (see also \cite[Theorem 6.3.5 (ii)]{fujino-foundations}). 
\end{proof}

Finally, we prove Proposition \ref{f-prop1.4}

\begin{proof}[Proof of Proposition \ref{f-prop1.4}]
We prove Conjecture \ref{f-conj1.3} under the extra assumption that 
Conjecture \ref{f-conj1.3} holds true for normal qlc pairs. 
Let $[X, \omega]$ be a qlc pair. 
Let $X_1$ be an irreducible component of $X$ and 
let $X_2$ be the union of the irreducible components 
of $X$ other than $X_1$. 
Then $X_1$, $X_2$, and $X_1\cap X_2$ are qlc pairs by adjunction 
(see Theorem \ref{z-thm2.6} and \cite[Theorem 6.3.5 (i)]{fujino-foundations}). 
In particular, they are seminormal (see \cite[Remark 6.2.11]{fujino-foundations}). 
Then we have the following short exact sequence 
$$
0\to \mathcal O_X\to \mathcal O_{X_1}\oplus \mathcal O_{X_2}\to 
\mathcal O_{X_1\cap X_2}\to 0
$$ 
(see, for example, \cite[Lemma 10.21]{kollar}). 
By \cite[Lemma 5.3.9]{fujino-foundations}, 
it is sufficient to prove Conjecture \ref{f-conj1.3} 
under the extra assumption that 
$X$ is irreducible by induction on $\dim X$ and 
the number of the irreducible components of $X$. 
Therefore, from now on, we assume that 
$X$ is irreducible. 
Let $\nu:Z\to X$ be the normalization. 
Then, by Theorem \ref{f-thm1.1}, 
we have 
\begin{equation}\label{f-eq3.1}
R\nu_*\mathcal I_{\Nqklt (Z, \nu^*\omega)}=\mathcal I_{\Nqklt (X, \omega)}. 
\end{equation} 
By induction on dimension, $\Nqklt (Z, \nu^*\omega)$ and 
$\Nqklt (X, \omega)$ are Du Bois since they are qlc (see Definition 
\ref{z-def2.7}). 
Since $Z$ is normal and $[Z, \nu^*\omega]$ is qlc, $Z$ is Du Bois by assumption. 
Therefore, by \cite[Corollary 6.28]{kollar} and 
\eqref{f-eq3.1}, $X$ is Du Bois. This is what we wanted. 
\end{proof}

We close this section with a remark on $\mathbb Q$-structures. 

\begin{rem}\label{f-rem3.1}
If $[X, \omega]$ has a $\mathbb Q$-structure in Theorem 
\ref{f-thm1.1}, 
then we can easily see that $[Z, \nu^*\omega]$ also has 
a $\mathbb Q$-structure by the proof of 
Theorem \ref{f-thm1.1}. 
\end{rem}

\end{document}